\documentclass[12pt]{amsart}
\usepackage{amssymb,amsmath,amsfonts,latexsym}
\usepackage{bm}
\usepackage{mathtools}
\usepackage{enumerate}
\usepackage{enumitem}
\usepackage{mathrsfs}
\usepackage{array,graphics,color}
\allowdisplaybreaks

\numberwithin{equation}{section}
\newtheorem{dfn}{Definition}[section]
\newtheorem{thm}[dfn]{Theorem}

\newtheorem{coro}[dfn]{Corollary}
\newtheorem{exm}[dfn]{Example}

\newtheorem{nte}[dfn]{Note}

\DeclarePairedDelimiterX{\norm}[1]{\lVert}{\rVert}{#1}
\DeclarePairedDelimiterX{\bnorm}[1]{\big\lVert}{\big\rVert}{#1}
\DeclarePairedDelimiterX{\Bnorm}[1]{\Big\lVert}{\Big\rVert}{#1}

\begin{document}
	
\title[Rough convergence of sequences in Controlled Metric Type Spaces]{Rough convergence of sequences in Controlled Metric Type Spaces}

\author[Khatun] {Sukila Khatun}
\address{Department of Mathematics, The University of Burdwan, Burdwan-713104, West Bengal, India}
\email{sukila610@gmail.com}
		
\author[Banerjee] {Amar Kumar Banerjee}
\address{Department of Mathematics, The University of Burdwan, Burdwan-713104, West Bengal, India}
\email{akbanerjee@math.buruniv.ac.in, akbanerjee1971@gmail.com}

\subjclass[2020]{40A05, 40A99}
	
\keywords{controlled metric type spaces, rough convergence, rough limit sets.}
	
\begin{abstract}
Mlaiki et al.\cite{MLA} introduced the idea of controlled metric type spaces, which is a new extension of $b$-metric spaces with addition of a controlled function $\alpha(x,y)$ of the right-hand side of the $b$-triangle inequality. Phu \cite{PHU} introduced the idea of rough convergence of sequences in a normed linear space. In this paper we have brought the idea of rough convergence of sequences in a controlled metric type space. We have proved several results associated with rough limit sets and some relevant results associated with such convergence. 
\end{abstract}
\maketitle
	
\section{Introduction}
Bakhtin \cite{BAK} introduced the concepts of $b$-metric spaces in 1989 and then this idea was used by many authors \cite{{ZER},{KAM}}. Here the triangle inequality is of the form $d(x,y) \leq s \{d(x,z)+d(z,y)\}$, where $s \geq 1$ is a given real number. Kamran et al. \cite{KAM} introduced the concept of extended b-metric spaces by replacing the constant $s$ by a function $\theta(x,y)$ depending on the parameters of the left-hand side of the triangle inequality i.e. $d(x,y) \leq \theta(x,y) \{d(x,z)+d(z,y)\} $. Recently, Mlaiki et al. \cite{MLA} introduced a new type of extended $b$-metric spaces, called controlled metric type spaces, replacing the function $\theta(x,y)$ by a control function $\alpha(x,y)$ in the right-hand side of the $b$-triangle inequality i.e. $d(x,y) \leq \{\alpha(x,z)d(x,z)+\alpha(z,y)d(z,y)\} $, for all $x,y,z \in X$. If $\alpha(x,y)=s \geq 1$,$\forall x,y \in X$, then $(X,d)$ reduces to a $b$-metric space. By this we can conclude that every $b$-metric spaces is controlled metric type spaces. But a controlled metric type space is not in general an extended $b$-metric space when taking the same function, i.e. $\theta = \alpha$. Mlaiki et al. \cite{MLA} cited an example of controlled metric type spaces that is not an extended $b$-metric space to show their extension is different. 
In 2001, H. X. Phu \cite{PHU} introduced the idea of rough convergence of sequences in a normed linear space. In 2003, Phu \cite{PHU1} extended this concept in an infinite dimensional normed space. After that many works were \cite{{AYTER1}, {AKB2}, {NH}, {PMROUGH1},{PAL}} done in more generalized spaces. The idea of rough convergence in a metric space was studied by S. Debnath and D. Rakhshit \cite{DR} and in a cone metric space was studied by A. K. Banerjee and R. Mondal \cite{RMROUGH}. 
In this paper we have brought this idea of rough convergence of sequences in controlled metric type spaces. We define the set of rough limit points and discussed several results associated with this set.

\section{Preliminaries}

\begin{dfn}\cite{PHU}
Let $\{ x_{n} \}$ be a sequence in a normed linear space $(X, \left\| . \right\|)$,  and $r$ be a non negative real number. Then $\{ x_{n} \}$ is said to be rough convergent to $x$ of roughness degree $r$ if for any $\epsilon >0$, there exists a natural number $k$ such that $\left\|x_{n} - x \right\| < r +\epsilon$ for all $n \geq k$.
\end{dfn}

\begin{dfn}\cite{MLA}
    Given a nonempty set $X$ and $\alpha: X \times X \longrightarrow [1, \infty)$. The function $d: X\times X \longrightarrow [0, \infty)$ is called a controlled metric type if \\
    $(d1)$ $d(x,y)=0$ if and only if $x=y$, \\
    $(d2)$ $d(x,y)=d(y,x)$, \\
    $(d3)$ $d(x,y) \leq \alpha(x,z)d(x,z)+\alpha(z,y)d(z,y)$, \\ 
    for all $x,y,z \in X$. The pair $(X,d)$ is called a controlled metric type space.
\end{dfn}
Let $(X,d)$ be a controlled metric type space. The open ball centred at $x$ with radius $r>0$ is given by 
     \begin{center}
         $B(x,r)=\{ y \in X : d(x,y)<r \}$.
     \end{center}
The closed ball $ \overline{B}(x,r)$ is given by
     \begin{center}
         $\overline{B}(x,r)=\{ y \in X : d(x,y) \leq r \}$.
     \end{center}

\begin{dfn}\cite{MLA}
    Let $(X,d)$ be a controlled metric type space and $\{ x_{n} \}_{n \in \mathbb N}$ be a sequence in $X$.\\
    Then the sequence $\{ x_{n} \}$ is said to be convergent to some $x$ in $X$ if for every $\varepsilon>0$, there exists $K=K(\varepsilon) \in \mathbb{N}$ such that $d(x_{n},x)<\varepsilon$ for all $n \geq K$. In this case, we write $lim_{n\to\infty}x_{n}=x$. \\
    The sequence $\{ x_{n} \}$ is said to be Cauchy if for every $\varepsilon>0$, there exists $K=K(\varepsilon) \in \mathbb{N}$ such that $d(x_{n},x_{m})<\varepsilon$ for all $n,m \geq K$. \\
    The controlled metric type space $(X,d)$ is said to be complete if every Cauchy sequence is convergent.
\end{dfn}

\section{Main Results}

\begin{dfn}(cf. \cite{PHU})
Let $(X,d)$ be a controlled metric type space and $r$ be a non-negative real number. A sequence $ \{ x_{n}\}$ in $X$ is said to be rough convergent to $x$ of roughness degree $r$ if for every $\varepsilon>0$ there exists a natural number $n_0$ such that $d( x_{n},x)<r+\varepsilon$ holds for all $n \geq n_0$.
\end{dfn}

Let $x_{n} \stackrel{r}{\longrightarrow} x$, then $r$ is said to be the degree of roughness of rough convergence of $\{x_n \}$. If $r=0$ the rough convergence becomes the ordinary convergence of sequences in controlled metric type spaces. Let $\{x_n \}$ is $r$-convergent to $x$, then $x$ is said to be a $r$-limit point of $\{x_{n} \}$. The set of all $r$-limit points of a sequence $\{x_{n} \}$ is said to be the $r$-limit set of the sequence $\{x_{n} \}$ and which is denoted by $LIM^{r}x_{n}$. So, $LIM^{r}x_{n}= \left\{x_{0} \in X : x_{n} \stackrel{r}{\longrightarrow} x \right\}$. We know that $r$-limit point of a sequence $\{x_{n} \}$ may not be unique.

By an example we can show that a sequence which is rough convergent in a controlled metric type space may not be  convergent in that space and also rough limit point may not be unique.

\begin{exm}
    Choose $X=\mathbb{N}=\{1,2,3.....\}$ and take $d: X\times X \longrightarrow [0, \infty)$ as
    \begin{equation*}
        \ d(x,y)=\begin{cases}
    0,                     & \leftrightarrow \text{  $x=y$}, \\
\frac{1}{\sqrt{x}},        & \text{if $x$ is even and $y$ is odd},  \\
\frac{1}{\sqrt{y}},        & \text{if $x$ is odd and $y$ is even},\\
  1,                       & otherwise.
    
    \end{cases}
    \end{equation*}
    Consider $\alpha: X\times X \longrightarrow [1, \infty)$ as
    \begin{equation*}
         \ \alpha(x,y)=\begin{cases}
\sqrt{x},   & \text{if  $x$ is even and $y$ is odd},  \\
\sqrt{y},   & \text{if  $x$ is odd and $y$ is even},\\
  1,        & otherwise.
    
    \end{cases}
    \end{equation*}
We will show that (X,d) is controlled metric type space. \\
\textbf{d(1):} $d(x,y)=0$ if and only if $x=y$ holds. \\
\textbf{d(2):} Let $x$ be even and $y$ be odd, then $d(x,y)=\frac{1}{\sqrt{x}}=d(y,x)$. \\
Similarly, if $x$ is odd and $y$ is even, then $d(x,y)=\frac{1}{\sqrt{y}}=d(y,x)$. \\ 
The other cases are similar. \\
\textbf{d(3):} we will prove this using some cases i.e. \\
\textbf{case(i):} If $z=x$, then $\alpha(x,z)d(x,z)+\alpha(z,y)d(z,y)$=$\alpha(x,x)d(x,x)+\alpha(x,y)d(x,y)$=$\alpha(x,x).0+\alpha(x,y)d(x,y)$=$\alpha(x,y)d(x,y) \geq d(x,y)$, since $\alpha(x,y) \geq 1$. \\
Again, if $z=y$, then $\alpha(x,z)d(x,z)+\alpha(z,y)d(z,y)$=$\alpha(x,y)d(x,y)+\alpha(y,y)d(y,y)$=$\alpha(x,y)d(x,y)+\alpha(y,y).0$=$\alpha(x,y)d(x,y) \geq d(x,y)$, since $\alpha(x,y) \geq 1$. \\
Hence $d(3)$ is satisfied. \\
\textbf{case(ii):} If $z \neq x, z \neq y$ and $x=y$ holds, then 
\begin{align*}
     \alpha(x,z)d(x,z)+\alpha(z,y)d(z,y)
&= \alpha(x,z)d(x,z)+\alpha(z,x)d(z,x)\\
&= \alpha(x,z)d(x,z)+\alpha(z,x)d(x,z)\\
&=  \{ \alpha(x,z)+\alpha(z,x)\}.d(x,z)\\
&= \begin{cases} 
        (\sqrt{x}+\sqrt{x}).\frac{1}{\sqrt{x}}=2,  \ \text{if $x=y$ is even and $z$ is odd}, \\
        (\sqrt{z}+\sqrt{z}).\frac{1}{\sqrt{z}}=2,  \ \text{if $x=y$ is odd and $z$ is even},
       \end{cases}
\end{align*}
Since $x=y, d(x,y)=0$. So, $(d3)$ holds. \\
\textbf{case(iii):} If $z \neq x$ and $z \neq y$ and also $x \neq y$ i.e. $x \neq y \neq z$, then we will prove $d(3)$ using some subcase. \\
\textbf{subcase(i):} Let $x,y,z$ are even, then 
       $\alpha(x,z)d(x,z)+\alpha(z,y)d(z,y)=1.1+1.1=2 \geq 1=d(x,y)$. So, $d(3)$ holds.\\
\textbf{subcase(ii):} Let $x,y,z$ are odd, then 
       $\alpha(x,z)d(x,z)+\alpha(z,y)d(z,y)=1.1+1.1=2 \geq 1=d(x,y)$. So, $d(3)$ holds. \\
\textbf{subcase(iii):} Let $x,z$ are even and $y$ be odd, then
       $\alpha(x,z)d(x,z)+\alpha(z,y)d(z,y)=1.1+\sqrt{z}.\frac{1}{\sqrt{z}}=1+1=2 \geq \frac{1}{\sqrt{x}}=d(x,y), \ \forall x \in \mathbb{N}$. \\
All other subcases are similar to the subcase(iii). \\
Hence $d$ is a controlled metric type and $(X,d)$ is a controlled metric type space. \\
Let us consider the sequence $ \{ \xi_n\}=\{2,3,2,3,......\}$ i.e. \begin{equation*}
    \{ \xi_n \}=\begin{cases}
2 & \text{ if $n$ is odd }\\
3 & \text{ if $n$ is even}

               \end{cases}
\end{equation*}
We show that $\{ \xi_n \}$ is not convergent in $(X,d)$. \\
Now, $lim_{n\to\infty}d( \xi_{2n},2)=lim_{n\to\infty}d(3,2)=\frac{1}{\sqrt{2}}$ and $lim_{n\to\infty}d(\xi_{2n+1},2)=lim_{n\to\infty}d(2,2)=0$.
So, $lim_{n\to\infty}d( \xi_{n},2)$ does not exist. So, $\{ \xi_{n} \}$ is not convergent to $2$. \\
Again, $lim_{n\to\infty}d( \xi_{2n},3)=lim_{n\to\infty}d(3,3)=0$ and $lim_{n\to\infty}d( \xi_{2n+1},3)=lim_{n\to\infty}d(2,3)=\frac{1}{\sqrt{2}}$.
So, $lim_{n\to\infty}d( \xi_{n},3)$ does not exist. So, $\{ \xi_{n} \}$ is not convergent to $3$. \\
Let $m \in \mathbb N, m \neq 2,3$. Then 
\begin{equation*}
   \ lim_{n\to\infty}d( \xi_{2n},m)=
    \ lim_{n\to\infty}d(3,m)=\begin{cases}
    \frac{1}{\sqrt{m}}, & \text{if $m$ is even}, \\
     1,                 & \text{if $m$ is odd}.
    
 \end{cases}
\end{equation*}
and \begin{equation*}
   \ lim_{n\to\infty}d( \xi_{2n+1},m)=
    \ lim_{n\to\infty}d(2,m)=\begin{cases}
1,                   & \text{if $m$ is even}, \\
\frac{1}{\sqrt{2}},  & \text{if $m$ is odd}.
    
 \end{cases}
\end{equation*}
Hence $\{ \xi_n \}$ is not convergent to any number of $X$ in $(X,d)$. \\
Now, we will show that $\{ \xi_n \}$ is rough convergent.
Let $\varepsilon>0$ be arbitrary. Then 
\begin{equation*}
    \ d(\xi_n,2)=\begin{cases}
d(2,2)=0<\frac{1}{\sqrt{2}}+\varepsilon                  & \text{ if $n$ is odd }\\
d(3,2)=\frac{1}{\sqrt{2}}<\frac{1}{\sqrt{2}}+\varepsilon & \text{ if $n$ is even }
    \end{cases}
\end{equation*}
and \begin{equation*}
    \ d(\xi_n,3)=\begin{cases}
d(2,3)=\frac{1}{\sqrt{2}}<\frac{1}{\sqrt{2}}+\varepsilon & \text{ if $n$ is odd }\\
d(3,3)=0<\frac{1}{\sqrt{2}}+\varepsilon                  & \text{ if $n$ is even }
    \end{cases}
\end{equation*} 
So, $\{ \xi_n \}$ is rough convergent of roughness degree $\frac{1}{\sqrt{2}}$ and in that $LIM^r \xi_{n}= \{2,3\}$. \\
And if $m\in \mathbb N$ and $m \neq 2,3$, then 
when $n$ is odd, 
\begin{equation*}
    \ d(\xi_n,m)= \ d(2,m) = \begin{cases}
    1<1+\varepsilon,                   & \text{if $m$ is even} \\
    \frac{1}{\sqrt{2}}<1+\varepsilon,  & \text{if $m$ is odd}
    \end{cases}
\end{equation*}
And when $n$ is even,  
\begin{equation*}
    \ d(\xi_n,m)= \ d(3,m) = \begin{cases}
    \frac{1}{\sqrt{m}}<1+\varepsilon, & \text{if $m$ is even} \\
    1<1+\varepsilon,                  & \text{if $m$ is odd}
    \end{cases}
\end{equation*}
Hence $\{ \xi_n \}$ is rough convergent to $m$ with roughness degree $r=1$ and $LIM^r \xi_{n}=\mathbb N$.  
\end{exm}

\begin{dfn} \cite{MLA}
    Let $(X,d)$ be a controlled metric type space. For a subset $A$ of $X$, the diameter of $A$ is defined by 
    \begin{center}
        $diam(A)=\sup \{d(x,y): x,y \in A \}$.
    \end{center}
\end{dfn}

In \cite{PHU} it is seen that the diameter of rough limit set in a normed linear space is $2r$, where as in $S$-metric space \cite{SUK} the diameter of rough limit set is $3r$ and in partial metric space \cite{SUK2} the diameter of rough limit set is $2r+2a$, where $r$ is roughness degree and $a$ is self-distance for all $x \in X$. 
The following theorem is a similar type of result.

\begin{thm}
If $\sup \alpha(x,y)$ exists and equal to $k$, then diameter of rough limit set is $2rk$.
\end{thm}
\begin{proof}
Let $x,y \in LIM^rx_{n}$ and $\varepsilon>0$ be arbitrary, and let 
    $\varepsilon_{1}=\frac{\varepsilon}{2k}$.
Then $\exists N \in \mathbb{N}, d( x_{n},x)<r+\varepsilon_{1}$ and $d( x_{n},y)<r+\varepsilon_{1}$ holds for all $n \geq N$.
Now, for all $n \geq N$, we can write 
\begin{equation*}
        \begin{split}
d(x,y) & \leq \alpha(x,x_{n})d(x,x_{n})+\alpha(x_{n},y)d(x_{n},y)\\
       & = \alpha(x,x_{n})d(x_{n},x)+\alpha(x_{n},y)d(x_{n},y)  \\
       & < k.(r+\varepsilon_{1})+k.(r+\varepsilon_{1}) \\
       & = 2k.(r+\varepsilon_{1}) \\
Therefore, \ sup \ d(x,y) & \leq 2k.(r+\varepsilon_{1}) \\
                          & = 2rk+\varepsilon_{1}.2k \\
                          & = 2rk+\varepsilon \\
        \end{split}
    \end{equation*}
Since $\epsilon>0$ is arbitrary, $diam(LIM^rx_{n}) \leq 2rk$............(i) \\
We now show that the diameter of rough limit set is not greater than $2rk$. \\
If possible, let $ \exists \ x,y \in LIM^rx_{n}$ such that $d(x,y) > 2rk$. \\
Let $\varepsilon^{'}>0$ be arbitrary and take $\varepsilon^{'}=\frac{\varepsilon}{2k}$, where $\varepsilon=d(x,y)-2rk$.
For this $\varepsilon$, we can find $n_1,n_2 \in \mathbb N$ such that 
$d(x_n,x)<r+\varepsilon^{'}$, $\forall n \geq n_1$
and $d(x_n,y)<r+\varepsilon^{'}$, $\forall n \geq n_2$. \\ 
Now, if $N=max \{ n_1,n_2 \}$, then 
$d(x_n,x)<r+\varepsilon^{'}$
and $d(x_n,y)<r+\varepsilon^{'}$, $\forall n \geq N$. \\ 
For $n \geq N$ and using the triangle inequality for controlled metric type spaces, we can write \\
\begin{equation*}
        \begin{split}
d(x,y) & \leq \alpha(x,x_{n})d(x,x_{n})+\alpha(x_{n},y)d(x_{n},y)\\
       & = \alpha(x,x_{n})d(x_{n},x)+\alpha(x_{n},y)d(x_{n},y)  \\
       & < k.(r+\varepsilon^{'})+k.(r+\varepsilon^{'}) \\
       & = 2rk+2k\varepsilon^{'} \\
       & = 2rk+ \varepsilon  \ \text{( since $\varepsilon^{'}=\frac{\varepsilon}{2k}$ )} \\
       & =2rk+d(x,y)-2rk \\
       & =d(x,y), \ \text{which is a contradiction.}
        \end{split}
    \end{equation*}
Hence there does not exist elements $x,y \in LIM^rx_{n}$ such that $d(x,y)>2rk$ holds.\\
So, the diameter of a rough limit set can not be greater than $2rk$. 
Therefore, in view of (i) the diameter of the rough limit set is $2rk$.
\end{proof}

\begin{coro}
If $\sup \alpha(x,y)$ exists, then $LIM^rx_{n}$ is a bounded set.  
\end{coro}

\begin{thm}
If a sequence $\{x_n \}$ in a controlled metric type space $(X,d)$ converges to $x$ and if $sup$ $\alpha(x,y)$ exists and equal to $k$, then $\overline{B}(x,r) \subset LIM^{rk}x_{n} $ and $LIM^rx_{n} \subset \overline{B}(x,rk) $.
\end{thm}

\begin{proof}
    Given that $\{x_n \}$ converges to $x$.
    Let $y \in \overline{B}(x,r)$ and $\varepsilon>0$ be arbitrary.
    So, there exists a natural number $N$ such that $d(x_n,x)< \varepsilon_1$ for all $n \geq N$, where $\varepsilon_1=\frac{\varepsilon}{k}$ and $d(x,y) \leq r$.
    Hence for $n \geq N$, we can write 
    \begin{equation*}
    d(x_n,y) \leq \alpha(x_n,x)d(x_n,x) + \alpha(x,y)d(x,y) \\
             < k.\varepsilon_1 + k.r \\
             = rk+\varepsilon
    \end{equation*}
    This implies that $ y \in LIM^{rk}x_{n} $.
    Hence $\overline{B}(x,r) \subset LIM^{rk}x_{n} $. \\
Now, let $z \in LIM^rx_{n}$.
Let $\varepsilon>0$ be arbitrary and let $\varepsilon_1= \frac{\varepsilon}{2k} $.
We can find $N_1,N_2 \in \mathbb N $ such that $ d(x_n,x)< \varepsilon_1,\ \forall n \geq N_1 $ and $ d(x_n,z)< r+\varepsilon_1,\ \forall n \geq N_2 $.
If $ N=max \{ N_1,N_2\}$, then $d(x_n,x)< \varepsilon_1$ and $d(x_n,z)< r+\varepsilon_1,\ \forall n \geq N $.
Now, for $n \geq N$ and using the triangle inequality for controlled metric type spaces, we can write 
\begin{equation*}
    \begin{split}
        d(x,z) & \leq \alpha(x,x_n)d(x,x_n) + \alpha(x_n,z)d(x_n,z) \\
               & = \alpha(x,x_n)d(x_n,x) + \alpha(x_n,z)d(x_n,z) \\
               & < k.\varepsilon_1 + k.(r+\varepsilon_1) \\
               & = rk + 2k.\varepsilon_1 \\
               & = rk + \varepsilon 
    \end{split}
\end{equation*}
Thus we have $d(x,z) \leq rk+\varepsilon$, for arbitrary $\varepsilon>0$.
 Hence $d(x,z) \leq rk$ holds. Therefore, $ z \in \overline{B}(x,rk)$. So, $LIM^rx_{n} \subset \overline{B}(x,rk) $. 
    
\end{proof}

\begin{thm}
The controlled metric type space $(X,d)$ is first countable. 
\end{thm}

\begin{proof}
    Consider $u=\{ B(x,\frac{1}{m}): x \in X, m \in \mathbb N \}$. Then $u \subset v$, where $v=\{ B(x,\varepsilon): x \in X, \varepsilon>0 \}$, the basis of the topology $\tau(d)$.
    Clearly, $u$ is a basis for $\tau(d)$. For, let $A \in \tau(d)$ and $x \in A$ be arbitrary. Then, since $v$ is a basis for $\tau(d)$, \ $\exists$  $\varepsilon>0$ such that $x \in B(x, \varepsilon) \subset A$. Choose $m \in \mathbb N$, so that $\frac{1}{m}<\varepsilon$. Then $B(x,\frac{1}{m}) \subset B(x, \varepsilon)$. Thus $x \in B(x,\frac{1}{m}) \subset B(x, \varepsilon) \subset A$ and so $u$ forms a basis for $\tau(d)$. Therefore $(X,d)$ is first countable, since $u$ is countable.
\end{proof}

\begin{thm}
    Let $\{x_n \}$ be a sequence in a controlled metric type space $(X,d)$ and if $\sup \alpha(x,y)$ exists for all $x,y \in X$ and equal to $k$, then $(LIM^rx_{n})^{'} \subset LIM^{rk}x_{n}$, where $(LIM^rx_{n})^{'}$ denote the set of all limit points of $LIM^rx_{n}$.
\end{thm}

\begin{proof}
    We consider the case when $LIM^rx_{n} \neq \phi$. Let $y \in (LIM^rx_{n})^{'}$. Then, let $\{y_n \}$ be a sequence in $LIM^rx_{n}$ converging to $y$. We will show that $ y \in LIM^{rk}x_{n}$.
    Let $\varepsilon>0$ be arbitrary positive real number and $\varepsilon_1 = \frac{\varepsilon}{2k}$.
    Now, we can find a $N \in \mathbb N$ such that $ d(y_n,y)<\varepsilon_1 $ and $ d(x_n,y_p)< r+\varepsilon_1 $ holds for every $n \geq N$, where $p$ is a fixed natural number greater than $N$. 
    Now, we can write 
    \begin{equation*}
        \begin{split}
    d(x_n,y) & \leq \alpha(x_n,y_p) d(x_n,y_p) + \alpha(y_p,y) d(y_p,y) \\
             & < k.(r+\varepsilon_1) + k.\varepsilon_1 \\
             & =rk + 2k.\varepsilon_1 \\
             & =rk + \varepsilon
        \end{split}
    \end{equation*}
    Hence $ y \in LIM^{rk}x_{n}$.
    So, $(LIM^rx_{n})^{'} \subset LIM^{rk}x_{n}$.
\end{proof}

\begin{nte}
   If $\alpha(x,y)=1$, $\forall x,y \in X$, then by above theorem, $(LIM^rx_{n})^{'} \subset LIM^{r}x_{n}$ which imply that $LIM^{r}x_{n}$ is a closed set. It is also seen in \cite{PHU} that $LIM^{r}x_{n}$ is closed. 
\end{nte}

In $(X, d)$ a sequence $\{ x_{n}\}$ is said to be bounded if and only if there exists a $B(>0)\in \mathbb{R}$ such that $d(x_{n}, x_{m}) < B$ for all $m,n \in \mathbb{N}$. The following theorem is a generalization of the classical property of a sequence that a convergent sequence must be bounded. 

\begin{thm}
    If $sup$ $\alpha(x,y)$ exists and equal to $k$, then every $r$-convergent sequence in controlled metric type spaces $(X,d)$ is bounded. 
\end{thm}

\begin{proof}
    Let $\{x_n \}$ be a sequence in a controlled metric type spaces $(X,d)$ $r$-convergent to $x$.
    We will show that $\{x_n \}$ is bounded in $X$.
    Now, for arbitrary $\varepsilon>0$, we can find a natural number $p$ such that $ d(x_n,x)< r+\varepsilon, \ \forall n \geq p $.
    Now, we consider $ M= max_{1 \leq i, j \leq p} \{ d(x_i,x_j)\}$.
    Now, we consider three cases \\
    \textbf{Case(i):} Let $i \leq p$ and $ j \geq p$, then 
    \begin{equation*}
        \begin{split}
    d(x_j,x_p) & \leq \alpha(x_j,x)d(x_j,x) + \alpha(x,x_p) d(x,x_p)   \\
               & < k.(r+\varepsilon) + k.(r+\varepsilon) \\
               & = 2k(r+\varepsilon)
        \end{split}
    \end{equation*}
Also, \begin{equation*}
    \begin{split}
d(x_i,x_j) & \leq \alpha(x_i,x_p)d(x_i,x_p) + \alpha(x_p,x_j) d(x_p,x_j)   \\  
           &  < k.M + k.2k(r+\varepsilon) \\
           & = kM + 2k^2(r+\varepsilon).
    \end{split}
\end{equation*}
\textbf{Case(ii):} Let $i \geq p$ and $j \leq p$, then interchanging the role of $i$ and $j$ in case(i), we can see
$d(x_i,x_j) \leq kM + 2k^2(r+\varepsilon)$. \\
\textbf{Case(iii):} Now, let $i \geq p$ and $j \geq p$, then $d(x_i,x_p) \leq 2k(r+\varepsilon)$ and $d(x_j,x_p) \leq 2k(r+\varepsilon)$.
Also, \begin{equation*}
    \begin{split}
d(x_i,x_j) & \leq \alpha(x_i,x_p)d(x_i,x_p) + \alpha(x_p,x_j) d(x_p,x_j)   \\
           & < k.2k(r+\varepsilon) + k.2k(r+\varepsilon) \\
           & = 4k^2 (r+\varepsilon).
    \end{split}
\end{equation*}
If $ B = max \{M, kM + 2k^2(r+\varepsilon), 4k^2 (r+\varepsilon) \}$, then $ d(x_i,x_j)<B, \ \forall i,j \in \mathbb N$.
Therefore, $\{x_n \}$ is bounded in $X$.
\end{proof}

In $\mathbb R$, we know that a bounded sequence may not be convergent. But the following result shows that a bounded sequence is always rough convergent.

\begin{thm}
    If $\sup \alpha(x_i,y_j)$ exists and equal to $k$, then a bounded sequence in controlled metric type spaces $(X,d)$ is always $r$-convergent for some degree of roughness $r=2kB$, where $B(>0) \in \mathbb R$. 
\end{thm}

\begin{proof}
   Let $\{x_n \}$ be a bounded sequence in controlled metric type spaces $(X,d)$. 
   So, there exists real number $B$ such that $ sup \{ d(x_n,x_m)\}< B, \ \forall n,m \in \mathbb N$.
   Let $p \in \mathbb N$ be fixed. Then for all $n \in \mathbb N$, we can write 
   \begin{equation*}
       d(x_n,x_p)  \leq \alpha(x_n,x_m)d(x_n,x_m) + \alpha(x_m,x_p) d(x_m,x_p)   
             < k.B+k.B 
             = 2kB
            < 2kB + \varepsilon, \ \forall n. 
  \end{equation*}
Hence $\{x_n \}$ rough converges to $x_p$ for degree of roughness $r=2kB$. 
\end{proof}

\begin{thm}
    Let $\{ x_{{n}_{i}}\}$ be a subsequence of $\{ x_{n}\}$, then  $LIM^{r} x_{n} \subseteq  LIM^{r} x_{{n}_{i}}$.
\end{thm}

\begin{proof}
 For arbitrary $\varepsilon>0$ and let $x \in LIM^{r} x_{n}$.
 Then there exists $m \in \mathbb N$ such that $d(x_n,x)<r+\varepsilon$, $\forall n \geq m$.
 Let $n_j>m$ for some $j \in \mathbb N$ and $n_i>m$ for $i \geq j$.
 Therefore $d(x_{{n}_{i}},x)<r+\varepsilon$, for for $i \geq j$.
 So, $x \in LIM^{r} x_{{n}_{i}} $.
 Hence $LIM^{r} x_{n} \subseteq  LIM^{r} x_{{n}_{i}}$.
\end{proof}

\begin{thm}
    Let $\sup \alpha(x,y)$ exists and equal to $k$. Also, let $\{a_n \}$ and $\{b_n \}$ be two sequences in controlled metric type spaces $(X,d)$ with the property that $d(a_i,b_i) \leq \frac{r}{k}, \ \forall i \geq p_1 $, for some natural number $p_1$ and $r>0$. Then $\{a_n \}$ converges to $\xi \in X$ imply that $\{b_n \}$ is converges to $\xi$.   
\end{thm}

\begin{proof}
   Let $\varepsilon>0$ be arbitrary and let $\{a_n \}$ converge to $\xi$. So, corresponding to this $\varepsilon>0$, we can find $p_2 \in \mathbb N$ such that $d(a_n,\xi)< \frac{\varepsilon}{k}, \ \forall n \geq p_2$.
   If we consider $ p=max \{ p_1,p_2 \}$, then for $n \geq p$ we have 
   \begin{equation*}
       d(b_n,\xi) \leq \alpha(b_n,a_n)d(b_n,a_n) + \alpha(a_n,\xi) d(a_n,\xi)
                   < k.\frac{r}{k} + k. \frac{\varepsilon}{k}
                   = r+\varepsilon, \ \forall n \geq p.
   \end{equation*}
Hence the result follows.
\end{proof}

\begin{thm}
    Let $\{x_n \}$ be a sequence in a controlled metric type space $(X,d)$ and $r$-convergent to $x$. Also, let $\sup \alpha(x,y)$ exists and equal to $k$. If $\{\xi_n \}$ is a convergent sequence in $LIM^r x_n$ converging to $\xi$ then $\{x_n \}$ is $rk$-convergent to $\xi$ in $X$. 
\end{thm}

\begin{proof}
    Let $\varepsilon>0$ be arbitrary and let $\varepsilon_1=\frac{\varepsilon}{2k}$.
    Since, $\{\xi_n \}$ converges to $\xi$, so there exists a natural number $p_1$ such that $d(\xi_n,\xi)<\varepsilon_1, \ \forall n \geq p_1$.
    Also, since $\{x_n \}$ is $r$-convergent to $x$, there exists a natural number $p_2$ such that $d(x_n,x)< r+\varepsilon_1, \ \forall n \geq p_2$.
    Let $p=max \{p_1,p_2 \}$ and consider a member $\xi_m$ of $\{\xi_n \}$ where $m >p$.
    Now, for all $n \geq p$, we have 
    \begin{equation*}
        d(x_n,\xi) \leq \alpha(x_n,\xi_m)d(x_n,\xi_m) + \alpha(\xi_m,\xi)d(\xi_m,\xi) 
                    < k.(r+\varepsilon_1) + k.\varepsilon_1
                    = rk + 2k.\varepsilon_1
                    =rk + \varepsilon
    \end{equation*}
Therefore, $d(x_n,\xi)< rk + \varepsilon$ holds for all $n \geq p$.
Hence the result follows.
\end{proof}

Let $\{x_{n}\}$ be a sequence in controlled metric type spaces $(X,d)$. Then $c \in X$ is said to be a cluster point of $\{x_{n}\}$, if for every $\varepsilon >0$ and every natural number $p$ there exists a natural number $m>p$ such that $d( x_{m},c)< \varepsilon$ holds.

\begin {thm}
Let $\{x_{n}\}$ be a $r$-convergent sequence in controlled metric type spaces $(X,d)$ and $\sup \alpha(x,y)$ exists and equal to $k$. Then for any cluster point $c$ of $\{x_{n}\}$, $LIM^{r}x_{n}  \subset \overline{B}(c,rk)$.
\end{thm}

\begin{proof}
 Let $\varepsilon>0$ be arbitrary and let $\varepsilon_1=\frac{\varepsilon}{2k}$.
 Let $x \in LIM^{r}x_{n}$. Then we can find a natural number $p$ such that $d(x_n,x)< r+\varepsilon_1, \ \forall n \geq p$.
 Also, since $c$ is a cluster point of $\{x_{n}\}$, there exists a natural number $m>p$ such that $d( x_{m},c)< \varepsilon_1$ holds.
 Now, for this natural number $m$, we can write 
 \begin{equation*}
     d(c,x) \leq \alpha(c,x_m)d(c,x_m) + \alpha(x_m,x)d(x_m,x) 
             < k.\varepsilon_1 + k. (r+\varepsilon_1)
             = k.\frac{\varepsilon}{2k} + k.(r+ \frac{\varepsilon}{2k})
             = r.k + \varepsilon
 \end{equation*}
Hence $ d(c,x)< rk+\varepsilon$.
Since $\varepsilon$ is arbitrary, $d(c,x) \leq rk$ i.e. $d(x,c) \leq rk$.
Therefore, $x \in \overline{B}(c,rk) $ and the result follows. 
\end{proof}

\subsection*{Acknowledgements}
The first author is thankful to The University of Burdwan for the grant of Senior Research Fellowship (State Funded) during the preparation of this paper. Both authors are also thankful to DST, Govt of India for providing FIST project to the deperment of Mathematics, B.U.


\begin{thebibliography}{99}
		
\bibitem{AYTER1}
S. Aytar, 
\textit{Rough statistical convergence,} 
Numer. Funct. Anal. Optim. {29(3-4)} (2008), 291-303.

\bibitem{AYTER2}
S. Aytar, 
\textit{The rough limit set and the core of a real Sequence,} 
Numer. Funct. Anal. Optim. {29(3-4)} (2008), 283-290.

\bibitem{BAK}
Bakhtin, I. A., 
\textit{The contraction mapping principle in almost metric spaces,}
Funct. Anal., 30, Unianowsk, Gos. Ped. Inst., (1989), 26-37.



\bibitem{AKB}
A. K. Banerjee and A. Dey, 
\textit{Metric Spaces and Complex Analysis,}  
New Age International (P) Limited, Publication, 
ISBN-10: 81-224-2260-8, ISBN-13: 978-81-224-2260-3.

\bibitem{RMROUGH}
A. K. Banerjee, R. Mondal,
\textit{Rough convergence of sequences in a cone metric space,} J. Anal. 27(3-4) (2019), 1179–1188.

\bibitem{SUK2}
A. K. Banerjee and S. Khatun, 
\textit{Rough convergence of sequences in a partial metric space,}
arXiv: 2211.03463, 2022.

\bibitem{AKB2}
A. K. Banerjee and A. Paul,
\textit{ On rough continuity and rough I-continuity of real functions},
arXiv: 2207.00542, 2022.

\bibitem{ZER}
S. Czerwik
\textit{Contraction mappings in b-metric spaces,}

Acta Math. Univ. Ostrav., 
{01(1)}(1993), 5-11.

\bibitem{DR}
S. Debnath and D. Rakshit, 
\textit{Rough convergence in metric spaces}, 
 In: Dang, P., Ku, M., Qian, T., Rodino, L. (eds) New Trends in Analysis and Interdisciplinary Applications. Trends in Mathematics(). Birkhäuser, Cham.$https://doi.org/10.1007/978-3-319-48812-7_57$. 
 
 \bibitem{NH}
 N. Hossain and A. K. Banerjee,
 \textit{Rough I-convergence in intuitionistie fuzzy normed space},
 Bulletin of Mathematical Analysis and Application,
 14(4) (20220), 1-10.
 
\bibitem{KAM}
T. Kamran, M. Samreen and Q. UL Ain,  
\textit{A Generalization of b-metric space and some fixed point theorems.}
Mathematics, 5(2)(2017), 1–7.

\bibitem{MLA}
N. Mlaiki, H. Aydi, N. Souayah and T. Abdeljawad, \textit{Controlled metric type spaces and the related
contraction principle,} 
Mathematics, {6(10)} (2018), 194-201.

\bibitem{PMROUGH1}
P. Malik, and M. Maity, 
\textit{On rough convergence of double sequence in normed linear spaces,}
Bull. Allahabad Math. Soc. {28(1)} (2013), 89-99.

\bibitem{PMROUGH2}
P. Malik, and M. Maity, 
\textit{On rough statistical convergence of double sequences in normed linear spaces,} 
Afr. Mat. 27(2016), 141-148.

\bibitem{SUK}
R. Mondal and S. Khatun, 
\textit{Rough convergence of sequences in an $S$ metric space,}
arXiv: 2204.04696, 2022.

\bibitem{PAL}
S. K. Pal, D. Chandra and S. Dutta,
\textit{Rough ideal convergence},
Hacet. J. Math. Stat. 42(6) (2013), 633–640.

\bibitem{PHU}
H. X. Phu, 
\textit{Rough convergence in normed linear spaces,} 
Numer. Funct. Anal. Optim.
22(1-2) (2001), 199-222.

\bibitem{PHU1}
H. X. Phu, 
\textit{Rough convergence in infinite dimensional normed spaces,} 
Numer. Funct. Anal. Optim. 24(2-3) (2003), 285-301.




\end{thebibliography}
\end{document}